\documentclass{amsart}
\usepackage{graphicx} 
\usepackage[left=1.15in,
            right=1.15in,
            top=1.15in,
            bottom=1.15in]{geometry}
\usepackage{setspace}
\usepackage{indentfirst}
\usepackage[utf8]{inputenc}
\usepackage{amsfonts}
\usepackage{amssymb}
\usepackage{amsmath, graphicx}
\usepackage{color}
\usepackage{amsthm}
\usepackage{amscd}
\usepackage{mathtools}
\usepackage{tikz}
\usetikzlibrary{positioning, fit, backgrounds, circuits.logic.US}
\usepackage{subcaption}
\usepackage[hidelinks]{hyperref}
\usepackage{cleveref}
\usepackage{lipsum}
\usepackage{algorithm}
\usetikzlibrary{cd}
\usepackage{faktor}
\usepackage{epigraph}
\usepackage{stmaryrd}
\usepackage{makecell}
\usepackage{blkarray}
\usepackage{ytableau} 
\usepackage[normalem]{ulem}
 
\newcommand{\N}{{\mathbb N}} 
 
\newcommand{\M}{{\mathcal M}}

\newcommand{\T}{{\mathcal T}}

\definecolor{choc}{RGB}{132,86,60}

\newtheorem{theorem}{Theorem}[section]
\newtheorem{lemma}[theorem]{Lemma}
\newtheorem{corollary}[theorem]{Corollary}
\newtheorem{proposition}[theorem]{Proposition}

\theoremstyle{definition}
\newtheorem{definition}[theorem]{Definition}

\newlength\tindent
\usepackage[hidelinks]{hyperref}

\title{The Ungar Games on Graded Posets}

\author{Jacob Paltrowitz}
\address{Harvard University}
\email{jacobpaltrowitz@college.harvard.edu}

\begin{document}

\begin{abstract}
    For a poset $P$, an Ungar move sends $P$ to $P\setminus T$, where $T$ is some subset of maximal elements of $P$. With these Ungar moves, Defant, Kravitz, and Williams define the Ungar games, where two players alternate making nontrivial Ungar moves until one player cannot make a move and loses. We characterize the second-player wins on graded posets. We first prove recursive characterizations of second-player wins before using these results to give classifications of the second-player wins in terms of boolean circuits. We also generalize Defant, Kravitz, and Williams' work on Young's Lattice $J(\N^2)$ to the higher-dimensional $J(\N^d)$. 
\end{abstract}

\maketitle

\section{Introduction}

In the game Nibble, defined in 2024 by Defant, Kravitz, and Williams \cite{defant2024ungargames}, players start with a rectangular chocolate bar and alternate nibbling off exposed top-right corners. The first player who cannot make a move loses. Two example nibbles are shown in \Cref{fig:nibbles}. In order to generalize Nibble, it is useful to think of the chocolate bar as a poset where each square of the chocolate bar is an element and a square is greater than the squares that lie weakly southwest of it. In this poset, the exposed top-right corners correspond to maximal elements. 

\begin{figure}
    \centering

    \scalebox{1.2}{
    \begin{tikzpicture}[baseline]
    \node at (-4.5, 0) {\ydiagram[*(choc)]{4,4,4,4}};
    \node at (-1,0) {\ydiagram[*(choc)]{3,4,4,4}};
    \node at (2.5,0) {\ydiagram[*(choc)]{2,3,4,4}};
    \node at (6, 0) {\ydiagram[*(choc)]{1, 2, 4, 4}};
    \node at (-2.75, 0) {$\rightarrow$};
    \node at (.75, 0) {$\rightarrow$};
    \node at (4.25, 0) {$\rightarrow$};
    \end{tikzpicture}}
    
    \caption{An example of three nibbles starting from a $4\times 4$ chocolate bar}
    \label{fig:nibbles}
\end{figure}
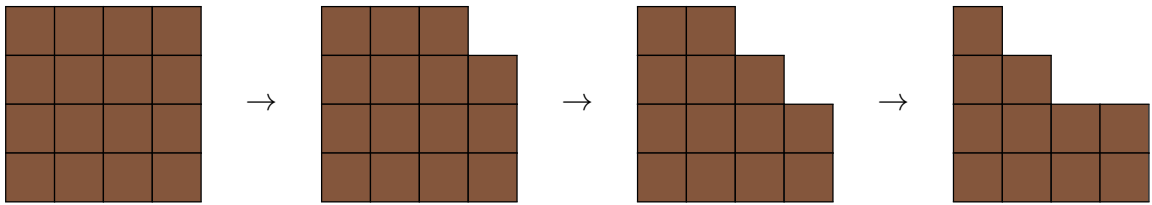

In 1982, Ungar \cite{UNGAR} solved a problem originally posed by Scott \cite{scottpoints} asking for the minimum number of possible slopes determined by a collection of non collinear points. In order to solve this problem, Ungar introduced a move between permutations which reverses consecutive decreasing subsequences. In 2023, Defant and Li \cite[Def. 1.3]{ungarianmarkov} generalized Ungar's construction on permutations to any meet-semilattice. Given some meet-semilattice $L$ and some element $x\in L$, let $\textrm{cov}_L(x)$ denote the set of elements covered by $x$.  An \textit{Ungar move} sends $x$ to the meet of some subset of elements of $\textrm{cov}_L(x)$. The moves originally constructed by Ungar correspond to the case where $L$ is the set of permutations of $n$ elements under the weak order. 

In \cite{defant2024ungargames}, Defant, Kravitz, and Williams specialize Ungar moves to  distributive lattices. Birkhoff's Representation Theorem \cite[Thm. 5]{Birkhoff1937} states that any finite distributive lattice $L$ is isomorphic to the poset $J(P)$ of order ideals of some poset $P$ ordered by containment. On a distributive lattice $J(P)$, the meet operation is equivalent to intersection of order ideals. Moreover, for $I\in J(P)$, we have that $\textrm{cov}_{J(P)}(I) = \{I\setminus\{x\}: x \in\max( I)\}$, where $\max(I)$ denotes the set of maximal elements of $I$. Thus, Ungar moves on distributive lattices correspond to removing maximal elements of a given order ideal, coinciding exactly with the moves in the game Nibble. 

\begin{definition}\cite[p.7]{defant2024ungargames}
    Let $J(P)$ be the distributive lattice formed from the finite order ideals of some poset $P$. Then for $X\in J(P)$ and $T\subseteq\max(X)$, an \textit{Ungar move} sends $X$ to $X\setminus T$. We call the elements of $T$ the \textit{ributes} of the Ungar move. If a move has no ributes, we call the move \textit{trivial}. 
\end{definition}

For the remainder of the paper, we will only consider this distributive lattice setting for Ungar moves. For studies of Ungar moves on other families of meet-semilattices, see \cite{defant2024ungargames} and \cite{ungarianmarkov}. 

Using Ungar moves, we can now generalize Nibble to the Ungar games. The \textit{Ungar games} consist of two players who alternate making nontrivial Ungar moves. Canonically the players are named Atniss and Eeta and Atniss makes the first move. The first player who cannot make a nontrivial move loses. Since this game has two players and perfect information, a given position is either a first or second player win under perfect strategy. We say a given poset $P$ is an \textit{Atniss win} if the first player has a winning strategy in the Ungar game starting at $P\in J(P)$ and an \textit{Eeta win} otherwise. 

In this setting, \cite[Thm. 1.5]{defant2024ungargames} classifies the Eeta wins in Nibble, corresponding to the Ungar games on elements of $J(\N^2)$. Given some partially nibbled chocolate bar $\lambda$, we can associate a string of $L$s and $D$s corresponding to the path from the top left corner of the bar to the bottom right. For example, the rightmost bar in \Cref{fig:nibbles} is associated with the sequence $LLDLDLDD$.  Defant, Kravitz, and Williams show that a chocolate bar $\lambda$ is an Eeta win if and only if the associated string of $L$s and $D$s never contains an odd length block of $L$s followed by an odd length block of $D$s. 

In this paper, we generalize the results of \cite{defant2024ungargames} to provide a full classification of Eeta wins for distributive lattices of graded posets. In particular, our results answer the open question of classifying the Eeta wins in $J(\N^d)$. Our work also specializes to more succinct and intuitive proofs of the results in \cite{defant2024ungargames}. In order to state the results, we first introduce some definitions. 

\begin{definition}
    Let $P$ be a poset. For any maximal element $m\in P$, the \textit{maximal subposet} at $m$, denoted $\M_P(m)$, is the subposet of $P$ comprised of all elements less than or equal to $m$ but no  other maximal element. An example is shown in \Cref{fig:maxposet}.
\end{definition}

One of our main results characterizes the relationship between posets and their maximal subposets in the Ungar games. 

\begin{theorem}
\label{bigthm}
    Let $P$ be a finite graded poset. Then $P$ is an Eeta win if and only if $\M_P(m)$ is an Eeta win for every maximal element $m\in P$.
\end{theorem}

Iterated application of this result allows for a stronger theorem, which requires the following definitions. 

\begin{definition}
    Let $P$ be a poset. The \textit{skeleton} of $P$, denoted $\T_P$, is the subposet of $P$ consisting of all elements that are not less than two incomparable elements of $P$. For some maximal element $m\in P$, the \textit{components} of the skeleton are $\T_P(m) = \T_P\cap \M_P(m)$. 
\end{definition}

\begin{definition}
    For some poset $P$ and element $x\in P$, let $G_P(x) = \{y\in P:y>x\}$. 
\end{definition}

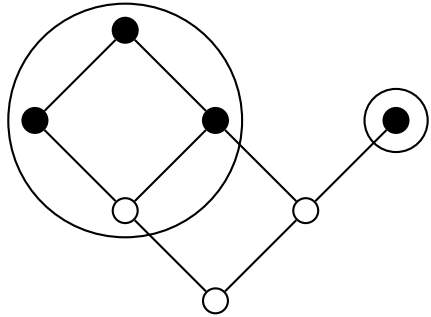
\begin{figure}[h!]
    \centering
\begin{tikzpicture}[scale=1.2,
    box/.style={draw, fill=white, minimum size=8mm, inner sep=0pt},
    every node/.style={circle, draw},
    >={Stealth[length=2pt]},
    thick]


\node[fill](a) at (0,0) {};

\node[fill] (b) at (1,-1) {};
\node[fill] (c) at (-1,-1) {};

\node (d) at (2,-2) {};
\node (e) at (0,-2) {};

\node[fill] (f) at (3,-1) {};
\node (g) at (1,-3) {};

\draw[-] (b) -- (a);
\draw[-] (c) -- (a);

\draw[-] (d) -- (b);
\draw[-] (e) -- (b);
\draw[-] (e) -- (c);

\draw[-] (f) -- (d);
\draw[-] (g) -- (d);
\draw[-] (g) -- (e);

\begin{scope}[on background layer]
  \node[thick, fit=(a)(b)(c)(e), inner sep=-8pt] {};
  \node[thick, fit=(f)] {};
\end{scope}

\end{tikzpicture}
    \caption{A poset whose maximal subposets are circled and whose skeleton is filled in with black.}
    \label{fig:maxposet}
\end{figure}

The following theorem utilizes \Cref{bigthm} to reduce the problem of determining winners of the Ungar games to considering only skeletons. 

\begin{theorem}
\label{subtreethm}
    A finite graded poset $P$ is an Eeta win if and only if its skeleton $\T_P$ is an Eeta win. 
\end{theorem}

It turns out that the skeleton has nice graph-theoretic structure. A \textit{rooted tree poset} is a poset with a single maximal element where the Hasse diagram is graph-theoretically a tree. Each skeleton consists of the disjoint union of rooted tree posets. Building upon \Cref{subtreethm}, it remains to classify which rooted tree posets are Eeta wins. This classification is done in the following theorem using the concept of boolean NAND formulas from theoretical computer science. See Section 2 for definitions and an example of NAND formulas. 

\begin{theorem}\label{treenand}
    Let $P$ be a finite rooted tree poset. Then the winner of the Ungar games on $P$ is equivalent to the output of the boolean \textrm{NAND} formula on $P$ on inputs of all $0$. 
\end{theorem}

Combining \Cref{subtreethm} and \Cref{treenand} gives a full classification of winning players of the Ungar games on graded posets in terms of NAND formulas. 

\begin{theorem}\label{biggestthm}
    Let $P$ be a finite graded poset. Then $P$ is an Eeta win if and only if $\T_P(m)$ evaluates to $1$ when considered as a boolean \textrm{NAND} formula on inputs of all $0$ for all $m$. 
\end{theorem}

We may apply the above results to elements of $J(\N^d)$ in order to answer one of the open questions in Section 6.1 of \cite{defant2024ungargames}. 

\begin{corollary}\label{ndeezthm}
    A poset $P\in J(\N^d)$ is an Eeta win if and only if for every maximal element $m\in P$, there is a maximal chain in $\T_P(m)$ of even length. 
\end{corollary}

Defant, Kravitz, and Williams conjecture that the Ungar games over distributive lattices are PSPACE complete. Such a result on the computational complexity of solving the Ungar games would give a notion of the intrinsic difficulty of classifying the winners of the Ungar games in this setting. From the classification of \Cref{biggestthm}, we are able to refute this conjecture in the graded setting. 

\begin{proposition} 
\label{logspace}
    The Ungar games on graded posets can be solved in logarithmic space. 
\end{proposition}

By the Space Hierarchy Theorem (see \cite{arora2006computational} for more details), there exist problems in PSPACE that are not in LOGSPACE. So, a problem in LOGSPACE cannot be PSPACE-complete. 

\section{Background}

We begin by introducing some useful definitions relating to partially ordered sets, or posets. A standard, more comprehensive, reference on posets is Chapter 3 of \cite{stanley1997enumerative}. The Ungar games  are played on finite posets; as such, all posets are assumed to be finite unless stated otherwise. Let $(P, \leq)$ be a poset. When the partial order is clear, we perform the standard abuse of notation and use $P$ to refer to $(P, \leq)$. An element $m\in P$ is \textit{maximal} if there is no $x\in P$ such that $x>m$. We denote the set of all maximal elements in $P$ by $\max(P)$. For any $x,y\in P$, if $x\leq y$ and there do not exist any $z\in P$ for which $x<z<y$, we say $y$ \textit{covers} $x$ and denote this by $x\lessdot y$. The set of elements in $P$ covered by $x$ is denoted $\textrm{cov}_P(x)$. Any subset of $P$ is considered a poset with the partial order inherited from $P$. A subset $I\subseteq P$ is an \textit{order ideal} if for any $x\in I$ we have that $y\leq x$ implies $y\in I$. The set of all finite order ideals of $P$ is denoted $J(P)$. The ideal generated by  $x_1, \dots, x_k\in P$ is $\langle x_1, \dots, x_k\rangle = \{y\in P: \exists x_i, y\leq x_i\}$. A subset of $S\subseteq P$ is \textit{convex} if for all $x,y\in S$ and $a\in P$, we have that $x\leq a\leq y$ implies $a\in S$. 

A poset is \textit{graded} if there exists a rank function $\rho:P\to \N$ where $\rho(x) = \rho(y)+1$ if $y\lessdot x$. Any convex subset of a graded poset is graded by the induced rank function. A \textit{chain} of $P$ is a totally ordered subposet of $P$. The \textit{length} of a chain is the number of elements in the subposet. The \textit{Cartesian product} of posets $(A, \leq_A)$ and $(B, \leq_B)$ is the poset on $A\times B$ with $(a, b)\leq (c,d)$ if and only if $a\leq_A c$ and $b\leq_B d$. Young's lattice is the poset $J(\N^2)$; Young's lattice can also be thought of as the poset of Young diagrams ordered under containment. The correspondence between order ideals in $\N^2$ and Young diagrams can be seen by drawing boxes around each point in $\N^2$, as shown in \Cref{fig:youngcorr}. We draw Young Diagrams in the French convention to align with the conventional orientation of $\N^2$. 

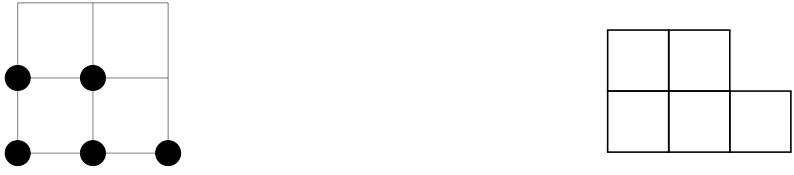
\begin{figure}
\label{fig:youngcorr}
  \centering

  \begin{subfigure}{0.49\textwidth}
    \centering
        \begin{tikzpicture}
    \draw[step=1cm,gray,very thin] (0,0) grid (2,2);
    \node[circle, draw, fill] at (0,0) {};
    \node[circle, draw, fill] at (1,0) {};
    \node[circle, draw, fill] at (0,1) {};
    \node[circle, draw, fill] at (1,1) {};
    \node[circle, draw, fill] at (2,0) {};
\end{tikzpicture}

  \end{subfigure}
  \hfill
  \begin{subfigure}{0.49\textwidth}
    \centering
    \scalebox{1.5}{
    \ydiagram{2, 3}}
    \vspace{5pt}
  \end{subfigure}

  \caption{An example of an ideal in $\N^2$ and the corresponding Young Diagram. }
\end{figure}


In order to state \Cref{treenand}, we introduce some notions from theoretical computer science. Denote the set of all finite binary strings as $\{0,1\}^*$.  The function $\textrm{NAND}:\{0,1\}^*\to\{0,1\}$ returns $0$ if and only if every letter in the input string is a $1$. A \textit{boolean NAND circuit} is a finite directed acyclic graph with a single sink node. On any assignment of $\{0,1\}$ to the source nodes, each target node is assigned a value corresponding to the NAND of the concatenation of the assignments of its source nodes. The output of the circuit on a given input is the value of the sink node. A circuit is called a \textit{formula} if every node has out-degree at most one. \Cref{fig:nandex} shows an example of a boolean NAND formula. See Chapter 6 of \cite{arora2006computational} for a more thorough treatment of circuits and related topics. 

\begin{figure}
    \centering
    \begin{tikzpicture}[every node/.style={font=\sffamily,align=center}, 
  gate/.style={circle, draw},
  input/.style={circle, draw, anchor=east},
  level distance=1.5cm,
  sibling distance=1.2cm]

  \node[input] (I1) at (0,3) {1};
  \node[input] (I2) at (0,2) {0};
  \node[input] (I3) at (0,1) {0};
  \node[input] (I5) at (0,-0.5) {0};
  \node[input] (I6) at (-3,3) {0};
  \node[input] (I7) at (3,-0.5) {1};

  \node[gate] (G1) at (3,2) {1};

  \node[gate] (G3) at (6,0.5) {0};

  \draw[->] (I1.east) -- (G1);
  \draw[->] (I2.east) -- (G1);
  \draw[->] (I3.east) -- (G1);


  \draw[->] (G1.east) -- (G3);
  \draw[->] (I5.east) -- (I7.west);
\draw[->] (I7.east) -- (G3);
\draw[->] (I6) -- (I1);


\end{tikzpicture}
    \caption{An example of boolean NAND formula where all inputs are $0$.}
    \label{fig:nandex}
\end{figure}
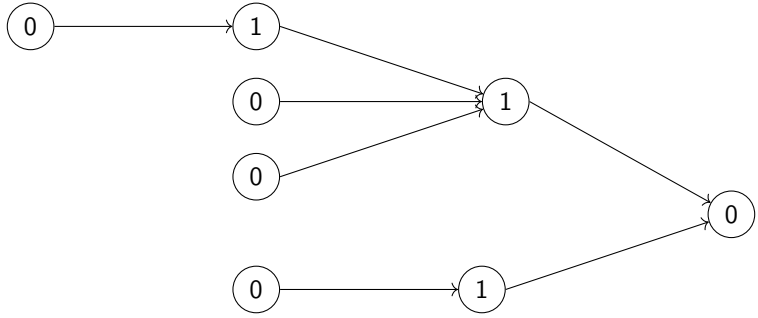

\section{Proof of Main Theorems}

We start with the proof of \Cref{bigthm}, which recursively characterizes the Eeta wins in terms of maximal subposets.

\begin{proof}[Proof of \Cref{bigthm}]
    We proceed by induction on $|P|$. The base case is for $|P| = 1$. This poset is an Atniss win and has one maximal subposet which is an Atniss win. We continue with the inductive step. Let $\rho$ denote a rank function on $P$. 

    For the forward direction, assume that $\M_P(m)$ is an Eeta win for each  $m\in\max(P)$. Let $T$ be a nonempty subset of $\max(P)$. Consider an arbitrary Ungar move sending $P$ to $P\setminus T$. Take some $t\in T$ of minimal rank. Since $\M_P(t)$ is an Eeta win, $P_t := \M_P(t)\setminus \{t\}$ must be an Atniss win. So, by induction, there must be some $c$ covered by $t$ in $\M_P(t)$ for which $\M_{P_t}(c)$ is an Atniss win. Note that $c$ is a maximal element of $P\setminus T$. We claim that $$\M_{P\setminus T}(c)=\M_{P_t}(c).$$ We have an inclusion $\M_{P_t}(c)\subseteq \M_{P\setminus T}(c)$. Note that $\M_{P\setminus T}(c)$ and $\M_{P_t}(c)$ are subsets of the order ideal $\langle c\rangle$. For any $x\in \langle c\rangle\setminus\{c\}$, then by the minimal rank assumption on $t$, for any $m\in \max(P)$, we have $\rho(x)\leq\rho(m)-2$. So, if $x$ is less than some maximal element of $P$ other than $t$, then $x$ is less than a maximal element of $P\setminus T$ other than $c$. Thus, any element $x\in \M_{P\setminus T}(c)$ cannot be less than any element of $\max(P)\setminus\{t\}$, giving the other inclusion.  
    
    Since $\M_{P_t}(c)$ is an Atniss win, so is $\M_{P\setminus T}(c)$. By induction, this means that $P\setminus T$ must be an Atniss win, as it has an Atniss win maximal subposet. Since $P\setminus T$ is an arbitrary Ungar move, $P$ must be an Eeta win. 

    To prove the reverse direction, we show that any poset with an Atniss win maximal subposet has a nontrivial Ungar move to an Eeta win poset. The proof of the converse relies on the following claim. Suppose $m\in\max(P)$ has minimal rank $k$ among the maximal elements where $\M_P(m)$ is an Atniss win. Then all maximal elements of $P\setminus\{m\}$ with rank less than $k$ have Eeta-win maximal subposets. 

    For notational purposes, let $c_1, \dots, c_n$ denote the elements for which $G_P(c_i) = \{m\}$. A maximal element of $P\setminus\{m\}$ of rank less than $k$ is either a maximal element of $P$ or one of the $c_i$. 

    We deal with these two cases separately. First, consider a maximal element $e\in \max(P)$ with rank less than $k$. By assumption, $\M_P(e)$ is an Eeta win. We claim that $\M_{P\setminus\{m\}}(e) = \M_P(e)$. We have an immediate inclusion $\M_P(e) \subseteq \M_{P\setminus\{m\}}(e)$. Now, consider some $x\in \M_{P\setminus\{m\}}(e)$. By definition, $x$ is not less than any maximal element of $P\setminus\{m\}$ other than $e$. So, it suffices to show that $x\not\leq m$. To see this, note that $\rho(x)\leq\rho(m) - 2$. If $x$ were less than $m$, then $x$ would have to be less than some element covered by $m$, and would then not be in $\M_{P\setminus\{m\}}(e)$. 

    For the second case, take some $c_i$ and consider $\M_{P\setminus\{m\}}(c_i)$. By assumption, $\M_P(m)$ is an Atniss win, so $P_m  = \M_P(m)\setminus\{m\}$ is an Eeta win. By the inductive hypothesis, every maximal subposet of $P_m$ is an Eeta win. So, it suffices to show that $\M_{P_m}(c_i) = \M_{P\setminus\{m\}}(c_i)$. There is a natural inclusion $\M_{P_m}(c_i)\subseteq\M_{P\setminus\{m\}}(c_i)$. For the other direction, we note that if $x\in \M_{P_m}(c_i)$, then $x$ is not less than any of the $c_j$, with $j\neq i$. Also, $x$ is not less than any other maximal elements of $P$, as $x\in P_m\subseteq \M_P(m)$. So, $x\in \M_{P\setminus\{m\}}(c_i)$. Thus $\M_{P\setminus\{m\}}(c_i)$ is an Eeta win as desired. 

    Having proved the claim, we can proceed with the reverse direction. Suppose there is some maximal element of $P$ with an Atniss win maximal subposet. We wish to show that $P$ is an Atniss win.  To do this, we construct a non-trivial Ungar move from $P$ to an Eeta win $P\setminus T$, for some $T\subseteq\max(P)$. We form $T$ iteratively as follows. To start, let $T_0 = \varnothing$. Then while $P\setminus T_{k-1}$ has an Atniss win maximal subposet, let $T_k = T_{k-1}\cup \{m_k\}$, where $m_k$ has minimal rank over all elements of $\max(P\setminus T_{k-1})$ for which $\M_{P\setminus T_k}(m_k)$ is an Atniss win. By assumption, $P$ has finitely many elements, so this process must terminate. We take $T$ to be the union of the $T_k$. 
    
    Having constructed $T$, we claim that the move from $P$ to $P\setminus T$ is the a valid Ungar move. By the claim proven between the directions of the proof, the sequence $\rho(m_1), \rho(m_2), \dots$ is non-decreasing. Next, we verify that the elements added to each $T_k$ are in fact maximal elements of $P$. This can be observed by the fact that the rank of each $m_k$ is at least the rank of $m_i$ for $i<k$ and $m_k$ is a maximal element of $P\setminus T_{k-1}$. Since the rank of $m_k$ is at least the rank of all of the elements in $T_{k-1}$, $m_k$ cannot be less than any element of $T_{k-1}$, and thus $m_k$ is a maximal element of $P$. When this process terminates, we have some $T$ such that the maximal subposets of $P\setminus T$ are all Eeta wins. By induction, this implies that $P\setminus T$ is an Eeta win, and $P$ is an Atniss win. 
\end{proof}

Now armed with \Cref{bigthm}, we look towards proving \Cref{subtreethm}. In order to do so, we first need the following facts about Eeta wins on unions and skeletons.

\begin{lemma}
    \label{unionwins}
    Let $P$ be a poset such that $P$ is the disjoint union of subposets $A_1, \dots, A_k$. Then $P$ is an Eeta win if and only $A_i$ is an Eeta win for all $i$. 
\end{lemma}

\begin{proof}
    The set of maximal subposets of $P$ is exactly the union of maximal subposets of the $A_i$. By \Cref{bigthm}, $P$ is an Eeta win if and only all of the maximal subposets are Eeta wins. So, $P$ is an Eeta win if and only if all of the $A_i$ have maximal subposets that are Eeta wins which is true if and only if all of the $A_i$ are Eeta wins.  
\end{proof}

\begin{lemma}\label{treeunion}
    For a poset $P$ we have $$\T_P = \bigsqcup_{m\in\max(P)}\T_P(m).$$
\end{lemma}

\begin{proof}
     Consider some $p\in \T_P$. Then since any two maximal elements are incomparable, $p$ is less than at most one maximal element. As $P$ is finite, $p$ is less than some maximal element. Thus, $p\in \M_P(m)$ and so $p\in \T_P(m)$ for some $m$. 

     Take some $p\in \T_P(m)$ for some $m\in \max(P)$. Then $p$ cannot be less than any element in $P\setminus\M_P(m)$. By definition, $p$ is not less than any two incomparable elements in $\M_P(m)$. So, $p\in \T_P$.
\end{proof}

\begin{lemma}\label{maximalskel}
    For a poset $P$ and a maximal element $m\in P$ we have $$\T_P(m) = \T_{\M_P(m)}.$$
\end{lemma}
\begin{proof}
    Consider some $x\in \T_P(m)$. By definition, $x\in \M_P(m)$ and $x$ is not less than any two incomparable elements in $P$. So, $x$ is not less than any two elements in $\M_P(m)$, and thus $x\in \T_{\M_P(m)}$. 

    Take some $y\in \T_{\M_P(m)}$. Observe that $y$ is not less than any element of $P\setminus \M_P(m)$. So, $y$ cannot be less than any two incomparable elements of $P$, and thus $y\in \T_P(m)$. 
\end{proof}

\begin{lemma}\label{maxelement}
    Let $P$ a poset with a single maximal element $m$, and let $c_1, \dots, c_k\lessdot m$ be the elements covered by $m$. Then $$\T_P = \{m\}\cup\bigsqcup_{i = 1}^k\T_{\M_{P\setminus\{m\}}(c_i)}.$$
\end{lemma}

\begin{proof}
    Note that all of the $c_i$ must be incomparable. So, any element $p\neq m$ that is not less than a two incomparable elements must be less than exactly one of the $c_i$. It follows that $p$ is in $\T(\M_{P\setminus\{m\}}(c_i))$ for some $i$. Also note that $m$ must be an element of $\T_P$. Finally, note that any element of $\M_{P\setminus\{m\}}(c_i)$ is not less than any element outside of $\M_{P\setminus\{m\}}(c_i)$ other than $m$, and thus any element in $\T_{\M_{P\setminus\{m\}}}(c_i)$ must be in $\T_P$. 
\end{proof}

We are now prepared to prove \Cref{subtreethm}, which gives the equivalence between posets and their skeletons in the Ungar games.

\begin{proof}[Proof of \Cref{subtreethm}]

    We prove this by induction on the number of elements in $P$. The base case of $|P|=1$ is tautological. By \Cref{bigthm}, it suffices to consider the maximal subposets of $P$. If $P$ has multiple maximal elements, then the maximal subposets have strictly smaller size than $P$. By induction, all of the maximal subposets $\M_P(m)$ are Eeta wins if and only if all of the $\T(\M_P(m))$ are Eeta wins. From  \Cref{unionwins} and \Cref{treeunion}, we know that all of the $\T(\M_P(m))$ are Eeta wins if and only if $\T_P$ is also an Eeta win. So, in the case that $P$ has multiple maximal elements, $P$ is an Eeta win if and only if $\T_P$ is an Eeta win.

    Otherwise, if $P$ has a single maximal element $m$, then there is only one nontrivial Ungar move, taking $P$ to $P\setminus\{m\}$. So, $P$ is an Eeta win if and only if $P\setminus\{m\}$ is an Atniss win. Since $P\setminus\{m\}$ has fewer than $|P|$ elements, we can apply the inductive hypothesis to say that $P\setminus\{m\}$ is an Atniss win if and only if $\T_{P\setminus\{m\}}$ is an Atniss win. By \cref{treeunion}, we have that $$\T_{P\setminus\{m\}} = \bigsqcup_{i=1}^k\T_{P\setminus m}(c_i).$$ So, by \Cref{maxelement}, $$\T_P = \{m\}\cup \T_{P\setminus\{m\}}.$$ Thus, $\T_{P\setminus\{m\}}$ is an Atniss win if and only if $\T_P$ is an Eeta win. Chaining together the equivalences, we see that $P$ is an Eeta win if and only if $\T_P$ is an Eeta win. 
    
\end{proof}

It is useful at this point to recall that the skeleton is graph-theoretically the union of rooted tree posets, as no element is covered by more than one element. Now, we prove \Cref{treenand}, characterizing the wins on tree shaped Hasse diagrams. 

\begin{proof}[Proof of \Cref{treenand}]
    On a rooted tree poset $P$, there is only a single Ungar move, given by removing the maximal element $m$. After this move, the remaining poset is the disjoint union of the subtrees rooted at the elements covered by $m$. By \Cref{unionwins}, $P$ is an Atniss win if and only if all of these subtrees are Eeta wins. Now, label each vertex $v\in P$ as follows: label with a $1$ if the order ideal $\langle v\rangle = \{x\in P:x\leq v\}$ is an Eeta win and $0$ otherwise. All of the leaves of $P$ are labeled $0$. Then under the above assignment, an element is labeled $0$ if and only if all of the elements it covers are labeled $1$. This aligns with the labeling of vertices given by the NAND function, and thus the winning player of the Ungar games on $P$ is given by the evaluation of the NAND formula on the rooted tree $P$. 
\end{proof}

Combining \Cref{subtreethm} and \Cref{treenand}, we can now prove \Cref{biggestthm} which classifies the winner of the Ungar games on a given poset in terms of a related NAND formula. 

\begin{proof}[Proof of \Cref{biggestthm}]
    By \Cref{subtreethm}, $P$ is an Eeta win if and only if $\T_P$ is an Eeta win. \Cref{treeunion} gives that $$\T_P = \bigsqcup_{m\in \max(P)}\T_P(m).$$ Thus, by \Cref{unionwins}, $\T_P$ is an Eeta win if and only if $\T_P(m)$ is an Eeta win for all $m$. Since $\T_P(m)$ is a rooted tree poset for all $m$, we can apply \Cref{treenand}; $\T_P(m)$ is an Eeta win if and only if the NAND formula on $\T_P(m)$ evaluates to $1$ on an input of all $0$. Chaining together the equivalences gives the desired result. 
\end{proof}

\section{Consequences of Main Theorems}

We can now explore the consequences of our main theorem. We start with a strengthening of Theorem~1.7 of \cite{defant2024ungargames}, which considers the Ungar games on quotients of order ideals.

\begin{theorem}
\label{subtractstuff}
    Let $P$ be a finite graded poset. Consider an $I\in J(P)$ such that each element of $I$ is less than two incomparable elements of $P$. Then $P$ is an Eeta win if and only if $P\setminus I$ is an Eeta win. 
\end{theorem}

\begin{proof}
     By \Cref{subtreethm}, it suffices to show that $\T_P = \T_{P\setminus I}$. Any element of $\T_P$ is not less than any two incomparable elements, and thus must be in $P\setminus I$. 
    
    In the other direction, consider an element $x\in \T_{P\setminus I}$. We have that $x$ is not less than any two incomparable elements in $P\setminus I$. Assume for the sake of contradiction that $x$ is less than some incomparable $a,b\in P$. By assumption, at least one of $a$ and $b$ is in $I$. Since $I$ is an order ideal and $x\leq a,b$, we can see that $x$ must then be in $I$, which contradicts the assumption that $x\in \T_{P\setminus I}$. 

    Since $\T_P$ and $\T_{P\setminus I}$ are isomorphic posets, \Cref{subtreethm} implies that $P$ is an Eeta win if and only if $P\setminus I$ is an Eeta win. 
\end{proof}

Applying \Cref{subtractstuff} to Young's lattice, we immediately get the following result, strengthening part of Theorem~1.5 of \cite{defant2024ungargames}. 

\begin{corollary}
    Consider $\lambda\in J(\N^2)$. Let $T\subset\lambda$ be the set of cells of $\lambda$ that have sides lying along the northeast boundary of $\lambda$, where $\lambda$ is drawn in the French convention. Then for any $\mu\in J(P)$ where $\mu\subseteq\lambda\setminus T$, the quotient $\lambda\setminus\mu$ is an Eeta win if and only if $\lambda$ is an Eeta win.
\end{corollary}

\begin{proof}
    Any element not on the northeast boundary is less than the element above it and the element to the right of it in the Young diagram. Therefore, each element of $\mu$ is less than two incomparable elements of $\lambda$. By  \Cref{subtractstuff}, $\lambda\setminus\mu$ is an Eeta win if and only if $\lambda$ is an Eeta win. 
\end{proof}

In order to prove \Cref{ndeezthm}, we first need a lemma describing the structure of skeletons of order ideals of cartesian products.  

\begin{lemma}
    \label{posetprod}
    Fix graded posets $(A_1, \leq_1), \dots, (A_k, \leq_k)$. For any $P\in J(\prod A_i)$ and $m\in\max(P)$, $\T_P(m)\setminus\{m\}$ is the disjoint union of posets isomorphic to convex subposets of some $A_i$. 
\end{lemma}

\begin{proof}[Proof of \Cref{posetprod}]
    Take some $P\in J(\prod A_i)$. Now, for some $a = (a_1, \dots,a_k)\in\max(P)$ consider $S = \T_P(a)\setminus\{a\}$. We claim that all of the connected components of $S$ are isomorphic to subsets of some $A_i$. To see this, consider a component $T$ of $S$. The maximum element of $T$ is covered by $(a_1, \dots,a_k)$ in $\prod A_i$, and thus must be of the form  $a' = (a_1, \dots, a_i',\dots, a_k)$,  with $a_i'\lessdot_i a_i$ for some $i\in [1,k]$. Then any element $x=(x_1,\dots, x_k)$ that is less than $a'$ must also be less than $x' = (x_1,\dots, a_i, \dots, x_k)$. By construction, we know that $x$ is not less than any two incomparable elements of $P$, so $a'$ must be less than $x'$. Thus, for $j\neq i$, we see that $a_j\leq x_j$ and so $a_j = x_j$. So, all elements in $T$ must have all $j$ components equal to $a_j$, and thus $S$ is isomorphic to some subtree in $A_i$. 

    To see that $T$ is convex, first note that $\T_P(a)$ is upwards closed, and thus convex. Since $a$ is a maximal element, $\T_P(a)\setminus\{a\}$ is convex as well. Since $T$ is a connected component of the Hasse diagram of $\T_P(a)\setminus\{a\}$, $T$ must be convex. 
\end{proof}

Using the above result, we can prove \Cref{ndeezthm}, answering the open question in Section~6.2 of \cite{defant2024ungargames} about the Eeta wins in $J(\N^d)$.

\begin{proof}[Proof of \Cref{ndeezthm}]
    To begin, observe that the only Ungar move on chains, or elements of $J(\N)$, consists of removing the lone maximal element of the chain. So, the Eeta wins in $J(\N)$ are exactly the chains with an even number of elements. Consider some $P\in J(\N^d)$. By \Cref{posetprod}, $$\T_P = \bigsqcup_{m\in\max(P)}\T_P(m),$$ where each $\T_P(m)$ consists of single element $m$ covering a collection of chains. By \Cref{unionwins}, $P$ is an Eeta win if and only if each $\T_P(m)$ is an Eeta win. In any $\T_P(m)$, removing $m$ leaves a disjoint union of chains. This union is an Atniss win if and only if any chain has odd length, implying that $\T_P(m)$ is an Eeta win if any of the original maximal chains has even length. 
\end{proof}

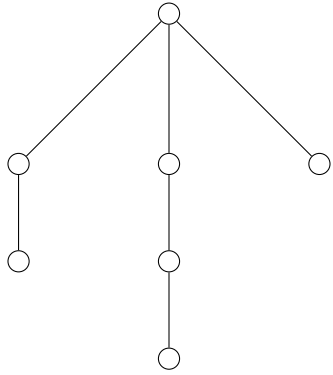
\begin{figure}
    \centering

\begin{tikzpicture}[every node/.style={circle,draw,minimum size=8pt,inner sep=2pt}, node distance=1cm and 2cm]

\node (a1) at (0,2) {};
\node (b1) at (2,2) {};
\node (c1) at (4,2) {};

\node (a2) [below=of a1] {};
\node (b2) [below=of b1] {};
\node (b3) [below=of b2] {};

\node (t) at (2,4) {};

\draw (a1) -- (a2);
\draw (b1) -- (b2) -- (b3);

\draw (a1) -- (t);
\draw (b1) -- (t);
\draw (c1) -- (t);
\end{tikzpicture}

    \caption{An example of a maximal subposet of an element of $J(\N^d)$.}
    \label{fig:ndeez}
\end{figure}



We may use \Cref{treenand} to make improvements in understanding the complexity of computing the winner of Ungar games on graded posets. Previously, the only known general method of determining the winner of the Ungar games on some poset $P$ was to recursively iterate upwards through all order ideals of $P$. This can be quite inefficient, as $|J(P)|$ can be exponential in $|P|$. As a result, Defant, Kravitz, and Williams conjectured that the Ungar games are PSPACE-complete. \Cref{logspace} refutes this in the graded case. 



\begin{proof}[Proof of \Cref{logspace}]
    Take some graded poset $P$. Determining the winner of the Ungar games on $P$ involves computing $\T_P$ and then evaluating the corresponding NAND circuit on $\T_P$. First, we show that in logarithmic space, one can determine whether a given element $x\in P$ is in $\T_P(m)$ for some fixed $m$. We first check whether $x\leq m$. Then we loop over all pairs of elements $a,b\in P$ and check whether $x\leq a, b$ and whether $a$ and $b$ are incomparable. By doing this, we check if $x\in \T_P(m)$. This entire process can be done in logarithmic space, as we only need to store a constant number of variables at a time. 

    Now it is known due to Corollary~4.1 of \cite{logspaceformulae} that a polynomial-sized formula can be evaluated in logarithmic space. Then using log-space transducers (see Lemma~4.15 of \cite{arora2006computational} for a more rigorous treatement), we know that the composition of these two steps can be performed in logarithmic space. 
\end{proof}

\section{Acknowledgments}

This research was conducted at the University of Minnesota Duluth REU with support from Jane Street Capital, NSF Grant 2409861, and donations from Ray Sidney and Eric Wepsic. I am grateful to Joe Gallian and Colin Defant for providing this wonderful opportunity. I would also like to thank Eliot Hodges, Noah Kravitz, Mitchell Lee, and Eric Shen for their helpful feedback and support.

\bibliographystyle{acm}
\bibliography{references}

\end{document}